\documentclass[12pt]{amsart}
\usepackage[margin=2.0cm]{geometry}

\usepackage{amssymb}
\usepackage{amsfonts}
\usepackage{mathrsfs}
\usepackage{cite}
\usepackage{graphicx}
\usepackage{mathtools}

\usepackage{float}

\usepackage{todonotes}

\usepackage{pgfplots}
\pgfplotsset{compat=1.15}
\usepackage{mathrsfs}
\usetikzlibrary{arrows}

\definecolor{ffqqqq}{rgb}{1,0,0}
\definecolor{qqzzqq}{rgb}{0,0.6,0}

\newcommand{\R}{{\mathbb R}}

\newtheorem{theorem}{Theorem}[section]

\newtheorem{lemma}[theorem]{Lemma}

\newtheorem{corollary}[theorem]{Corollary}
\newtheorem{conj}[theorem]{Conjecture}

\DeclareMathOperator{\arcosh}{\mathrm{arcosh}}
\DeclareMathOperator{\arsinh}{\mathrm{arsinh}}
\DeclareMathOperator{\artanh}{\mathrm{artanh}}

\DeclareMathOperator{\diam}{\mathrm{diam}}

\DeclareMathOperator{\conv}{\mathrm{conv}}

\DeclareMathOperator{\perim}{\mathrm{perim}}

\makeatletter
\DeclareFontFamily{U}{tipa}{}
\DeclareFontShape{U}{tipa}{m}{n}{<->tipa10}{}
\newcommand{\arc@char}{{\usefont{U}{tipa}{m}{n}\symbol{62}}}%

\newcommand{\arc}[1]{\mathpalette\arc@arc{#1}}

\newcommand{\arc@arc}[2]{%
  \sbox0{$\m@th#1#2$}%
  \vbox{
    \hbox{\resizebox{\wd0}{\height}{\arc@char}}
    \nointerlineskip
    \box0
  }%
}
\makeatother

\title[Ordinary reduced hyperbolic polygons]{On the perimeter, diameter and circumradius of ordinary hyperbolic reduced polygons}
\author{\'Ad\'am Sagmeister}
\address{Bolyai Institute, University of Szeged, Aradi vértanúk tere 1, H-6720 Szeged,
Hungary}
\email{sagmeister.adam@gmail.com }
\subjclass[2020]{Primary: 51M09, 51M10, 52A55}
\keywords{convex geometry, hyperbolic geometry, minimal width, thickness, reduced bodies, reduced polygons, perimeter, diameter, circumradius}

\mathtoolsset{showonlyrefs=true}

\begin{document}

\begin{abstract}
A convex body $R$ in the hyperbolic plane is called reduced if any convex body $K\subset R$ has a smaller minimal width than $R$. We answer a few of Lassak's questions about ordinary reduced polygons regarding its perimeter, diameter and circumradius, and we also obtain a hyperbolic extension of a result of Fabi\'nska.
\end{abstract}

\maketitle

\section{Introduction}
\label{sec:intro}

The concept of reducedness was introduced by Heil \cite{H78} in 1978 motivated by some volume minimizing problems. A convex body (i. e. a convex compact set of non-empty interior) $K$ is called reduced if an arbitrary convex body strictly contained in $K$ has smaller minimal width than $K$. P\'al \cite{Pal} proved in 1921 that for fixed minimal width, the regular triangle has minimal area among convex bodies in the Euclidean plane. This result of his is also known as the isominwidth inequality. The same problem in higher dimensions remains unsolved, as there are no reduced simplices in $\R^n$ for $n\geq 3$ (see \cite{MW02,MS04}), therefore there are no really good candidates for the volume minimizing problems -- so far the best one in $\R^3$ is the so-called Heil body, which has a smaller volume than any rotationally symmetric body of the same minimal width. The problem can be naturally generalized to other spaces, a natural approach is to study the problem in spaces of constant curvature. Bezdek and Blekherman \cite{Bez00} proved that, if the minimal width is at most $\frac{\pi}{2}$, the regular triangle minimizes the area in $S^2$. However, for spherical bodies of larger minimal width, the minimizers of the isominwidth problem are polars of Reuleaux triangles. Surprisingly enough, there is no solution of the isominwidth problem in the hyperbolic space for arbitrary dimension (this is part of an ongoing work joint with K. J. B\"or\"oczky and A. Freyer).

A reverse isominwidth problem is about finding the maximal volume if the minimal width is fixed. Naturally, this problem does not have a maximizer in general, but for reduced bodies we can ask for the convex body that maximizes the volume. However, in $\R^3$, the diameter of a reduced body of a given minimal width can be arbitrarily large, and hence the Euclidean problem is only interesting on the plane. It is conjectured, that the unique planar reduced bodies maximizing the area and of minimal width $w>0$ in $\R^2$ are the circular disk of radius $\frac{w}{2}$ and the quarter of the disk of radius $w$. A big step towards the proof of this conjecture was made by Lassak, who proved that among reduced $k$-gons, regular ones maximize the area, and as a consequence, all reduced polygons have smaller area than the circle. Following Lassak's footsteps, the same conclusion was derived in $S^2$ by Liu, Chan and Su \cite{LCS}.
Interestingly enough, the characterization of hyperbolic reduced polygons is still unclear, but clearly it must be different from the Euclidean and spherical ones; there exist reduced rhombi on the hyperbolic plane, while Euclidean and spherical reduced polygons are all odd-gons (see Lassak \cite{Las90,Las15}). However, the so-called ordinary reduced polygons can be examined the same way (these are odd-gons whose vertices have distance equal to the minimal width of the polygon from the opposite sides such that the projection of the vertices to these sides are in the relative interior of the sides). In fact, it was shown by the author \cite{S24} that, among ordinary reduced $n$-gons of a fixed width, regular $n$-gons maximize the area.
This answers one of Lassak's questions posed in \cite{Las24}. One of his other questions was about the extremality of the perimeter, which is addressed in Section \ref{sec:perim} where we give an explicit perimeter formula. However, the extremality of the perimeter remains open, but based on the given formula in Theorem \ref{thm:perimeter:hyp}, we propose a conjecture that is surprising given the Euclidean and spherical analogues.
%\begin{theorem}\label{thm:perim:hyp}
%On the hyperbolic plane among ordinary reduced $n$-gons of minimal width $w$, regular $n$-gons of the same minimal width have the smallest perimeter.
%\end{theorem}
%To the best of the author's knowledge, this theorem does not have a Euclidean counterpart yet, so we also prove the following:
%\begin{theorem}\label{thm:perim:eucl}
%On the Euclidean plane among reduced $n$-gons of minimal width $w$, regular $n$-gons of the same minimal width have the smallest perimeter.
%\end{theorem}
Lassak also proposed to find the infimum of the circumradii of ordinary reduced polygons of minimal width $w$. In order to obtain the best bound for the circumradius, we need the following result.
\begin{theorem}\label{thm:diameterbound:sagmeister}
If $P\subset H^2$ is an ordinary reduced polygon of minimal width $w$, then its diameter is at most
$$
\diam\left(P\right)\leq 2\arcosh\left(\frac{\cosh w+\sqrt{\cosh^2 w+8}}{4}\right)
$$
with equality if and only if $P$ is the regular triangle.
\end{theorem}
With this diameter bound, we obtain the following bound for the circumradius.
\begin{theorem}\label{thm:circumradiusbound}
Let $P$ be an ordinary reduced polygon in $H^2$ of minimal width $w$. Then its circumradius is at most
$$
\arsinh\left(\frac{2}{\sqrt{3}}\sqrt{\left(\frac{\cosh w+\sqrt{\cosh^2 w+8}}{4}\right)^2-1}\right)
$$
with equality if and only if $P$ is a regular triangle.
\end{theorem}
Finally, we also have the following result which is analogous to the results of Fabi\'nska \cite{Fab} and Musielak \cite{Mus}.
\begin{theorem}\label{thm:smallcircleattheboundary}
Every ordinary reduced hyperbolic polygon of minimal width $w$ is contained in a circular disk of radius $w$ centered at one of its boundary points.
\end{theorem}
The paper is structured as follows. In Section \ref{sec:preliminaries} we describe the fundamental concepts and notations about hyperbolic width and reducedness. In Section \ref{sec:ordinary}, we introduce  ordinary reduced polygons and explain some of their basic properties, including a few of the key ideas that will be used to obtain some of the main results. In Section \ref{sec:perim} we give a perimeter formula for ordinary reduced $n$-gons. In Section \ref{sec:diameter} we prove Theorem \ref{thm:diameterbound:sagmeister}, while in Section \ref{sec:circumradius} we prove Theorem \ref{thm:circumradiusbound} and Theorem \ref{thm:smallcircleattheboundary}.

\section{Preliminaries}\label{sec:preliminaries}

We use the notation $H^2$ for the hyperbolic plane, which is equipped with the geodesic metric. The geodesic distance of two points $x,y\in H^2$ will be denoted as $d\left(x,y\right)$. In this section, we introduce hyperbolic convexity. Many of the concepts are identical with their Euclidean analogues, but as we will soon see, there are exceptions.

For a subset $X$ of the hyperbolic plane $H^2$, we say that $X$ is \emph{convex}, if for any pair of points $x$ and $y$, the unique geodesic segment $\left[x,y\right]$ connecting $x$ and $y$ is a subset of $X$ (where $\left[x,x\right]=\left\{x\right\}$). A \emph{convex body} is a convex compact set of non-empty interior. It is clear, that similarly to Euclidean convexity, the intersection of an arbitrary family of convex sets in the hyperbolic plane is also convex, so we define the \emph{convex hull} of a set $X\subseteq H^2$ as the intersection of all convex sets in $H^2$ containing $X$ as a subset, and we will use the notation $\conv\left(X\right)$ for the convex hull of $X$. The convex body obtained as the convex hull of the finite set $X=\left\{x_1,\ldots,x_n\right\}$ is called a \emph{polygon}, and we use the notation $\left[x_1,\ldots,x_n\right]$ for $\conv\left(X\right)$. A point $x_j\in X$ is a \emph{vertex} of the polygon $X$ if $x_j\not\in\conv\left(X\setminus\left\{x_j\right\}\right)$; a \emph{$k$-gon} in the hyperbolic plane is a polygon of $k$ vertices.

For convex bodies, width is an important concept. On the hyperbolic plane there are many different notions of width (see Santal\'o \cite{S45}, Fillmore \cite{Fil70}, Leichtweiss \cite{Lei05}, Jer\'onimo-Castro--Jimenez-Lopez \cite{JCJL17}, G. Horv\'ath \cite{Hor21}, B\"or\"oczky--Cs\'epai--Sagmeister \cite{BoCsS}, Lassak \cite{Las24}). We will use the width function introduced by Lassak, but we note that it is identical with the extended Leichtweiss width defined by B\"or\"oczky, Cs\'epai and Sagmeister \cite{BoCsS} on supporting lines. A hyperbolic line $\ell$ is called a \emph{supporting line} of the convex body $K$ if $K\cap\ell\neq\emptyset$, and $K$ is contained in one of the closed half-spaces bounded by $\ell$. The \emph{width} of the convex body $K$ with respect to the supporting line $\ell$ is the distance of $\ell$ and $\ell'$, where $\ell'$ is a (not necessarily unique) most distant supporting line from $\ell$. It is known that this width function is continuous, and its maximal value coincides with the diameter of the convex body, which will be denoted as $\diam\left(K\right)$. The \emph{minimal width} (i.e. the minimal value of the width function on the set of all supporting lines, also known as the \emph{thickness}) of $K$ is denoted by $w\left(K\right)$. This notion of minimal width is a monotonic function, that is for arbitrary convex bodies $K,L$ such that $K\subseteq L$, we have $w\left(K\right)\leq w\left(L\right)$.  Hence, the concept of hyperbolic reducedness makes perfect sense. A convex body $K$ is called \emph{reduced}, if for any convex body $K'\subsetneq K$, $w\left(K'\right)<w\left(K\right)$. Reduced bodies are well-studied (see Heil \cite{H78}, Lassak--Martini \cite{LaM05,LaM11,LaM14}, Lassak--Musielak \cite{LaM18r}), as they are often extremizers of volume minimizing problems, and bodies of constant width are also reduced.

If we consider the Poincar\'e disk model of the hyperbolic plane $H^2$, hyperbolic lines are either diameters of the unit disk, or circular arcs intersecting the unit circle orthogonally. The boundary points of the unit disk are the \emph{ideal points} of the hyperbolic plane, and hence there is a natural bijection between hyperbolic lines and pairs of ideal points. Besides the identity, there are three types of orientation preserving isometries of the hyperbolic plane depending the number of fixed points. If there is one fixed point, the isometry is called an \emph{elliptic isometry} (or \emph{rotation}). We call an isometry with exactly one fixed ideal point, it is called a \emph{parabolic isometry}. Finally, isometries with exactly two fixed ideal points are called \emph{hyperbolic isometries}, which maps the line corresponding to the two fixed ideal points to itself.

\section{Ordinary reduced polygons}\label{sec:ordinary}

Lassak proved that hyperpolic convex odd-gons of thickness $w$ are reduced if all vertices are of distance $w$ from the opposite sides, and the orthogonal projections of these vertices onto the opposite sides are in the relative interior of these sides (see \cite{Las24}). Such polygons are called \emph{ordinary reduced polygons}, since this property characterizes reducedness both in $\R^2$ and in $S^2$ (see Lassak \cite{Las90,Las15}), but not in the hyperbolic plane. In an ongoing work with Ansgar Freyer and K\'aroly Jr. B\"or\"oczky we show that, for each $w>0$ there are reduced rhombi, whose diameters are unbounded. The characterization of hyperbolic reduced polygons is therefore unclear, so we focus on ordinary reduced polygons in this paper. For the diameter of an ordinary reduced polygon of thickness $w$, we have the following by Lassak \cite{Las24}.

\begin{theorem}\label{thm:diameterbound:lassak}
Let $P\subset H^2$ be an ordinary reduced polygon of thickness $w$ and diameter $d$. Then,
$$
w<d<\arcosh\left(\cosh w\sqrt{1+\frac{\sqrt{2}}{2}\sinh w}\right).
$$
\end{theorem}

As a consequence, for each $n$ we can expect an $n$-gon of extremal area among ordinary reduced $n$-gons of thickness $w$ by Blaschke's Selection Theorem. In the remainder of the section we will discuss the area of hyperbolic reduced $n$-gons based on the arguments of Lassak \cite{Las05} and Liu--Chang--Su \cite{LCS}.

From now on, $P$ denotes an ordinary reduced $n$-gon in $H^2$ whose vertices are $v_1,\ldots,v_n$ in cyclic order with respect to the positive orientation. For each $i$, let $t_i$ be the orthogonal projection of $v_i$ on the line through $v_{i+\frac{n-1}{2}}$ and $v_{i+\frac{n+1}{2}}$, where the indices are taken mod $n$. By definition, $t_i$ is in the relative interior of $\left[v_{i+\frac{n-1}{2}},v_{i+\frac{n+1}{2}}\right]$, and hence the geodesic segments $\left[v_i,t_i\right]$ and $\left[v_{i+\frac{n+1}{2}},t_{i+\frac{n+1}{2}}\right]$ intersect in a point $p_i$. Let $B_i$ be the union of the two triangles:
$$
B_i=\left[v_i,p_i,t_{i+\frac{n+1}{2}}\right]\cup\left[v_{i+\frac{n+1}{2}},p_i,t_i\right];
$$
we will call $B_i$ a \emph{butterfly}. We observe that these butterflies cover the polygon (see Sagmeister \cite{S24}).

\begin{center}
\includegraphics[scale=0.7]{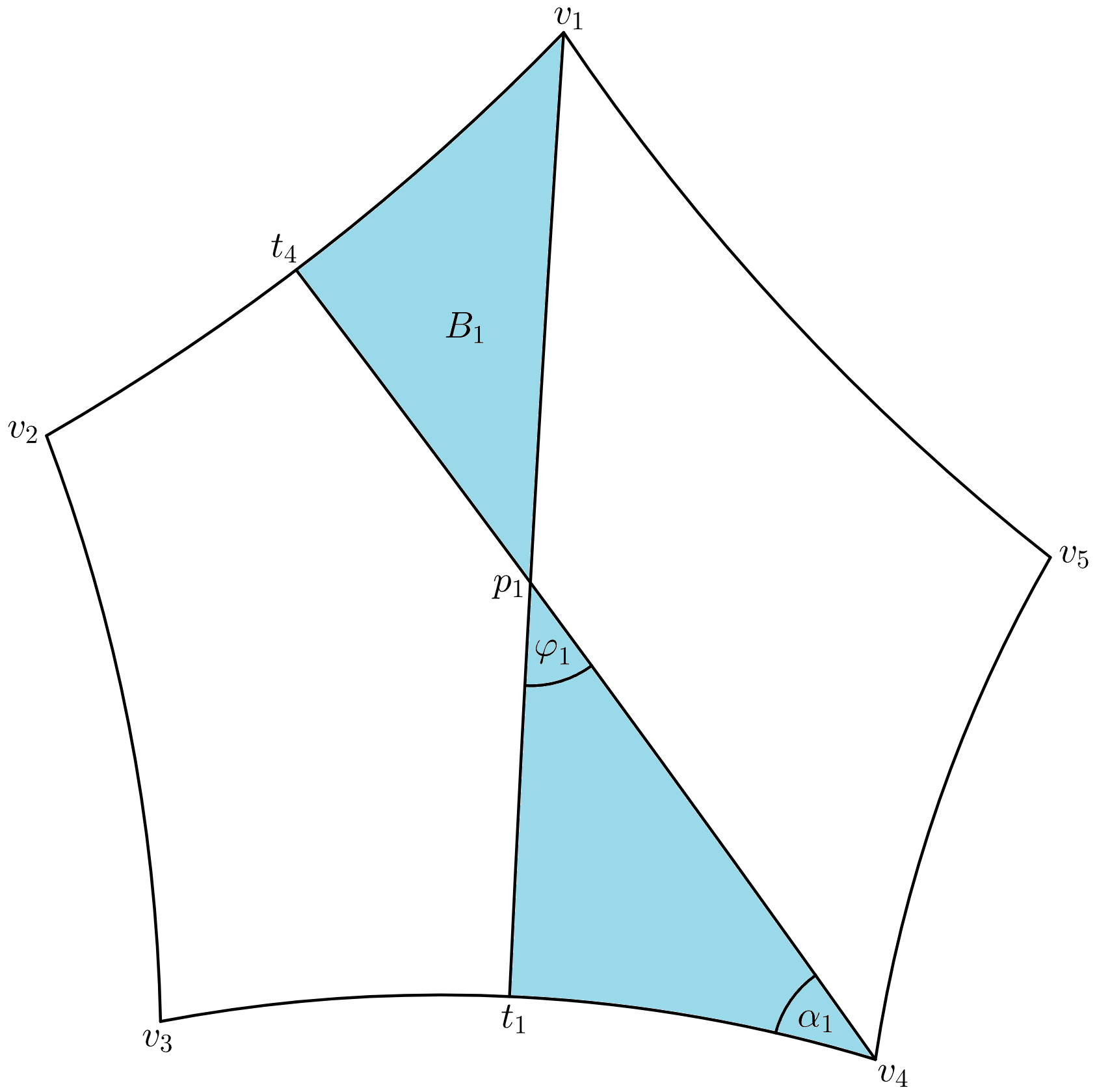}
\end{center}

\begin{lemma}\label{lemma:butterflies}
Let $P\subset H^2$ be an ordinary reduced $n$-gon, and $B_i$ be its $i^{\mathrm{th}}$ butterfly. Then,
$$
P=\bigcup_{i=1}^n B_i.
$$
\end{lemma}

We introduce a few additional notations for some angles of the butterflies. Let
$$
\varphi_i=\angle\left(v_i,p_i,t_{i+\frac{n+1}{2}}\right)=\angle\left(t_i,p_i,v_{i+\frac{n+1}{2}}\right)
$$
and
$$
\alpha_i=\angle\left(t_i,v_{i+\frac{n+1}{2}},p_i\right).
$$
The following lemma from Sagmeister \cite{S24} shows that the two triangles involved in the butterfly $B_i$ are congruent.

\begin{lemma}\label{lemma:congruence}
The two triangles $\left[v_i,p_i,t_{i+\frac{n+1}{2}}\right]$ and $\left[v_{i+\frac{n+1}{2}},p_i,t_i\right]$ defining $B_i$ are congruent.
\end{lemma}

Another observation from Sagmeister \cite{S24} provides an upper bound for the sum of the vertical angles of the butterflies.

\begin{lemma}\label{lemma:anglesum}
For an ordinary reduced $n$-gon $P$ with the notations from above,
$$
\sum_{i=1}^n
\varphi_i\leq\pi.
$$
%with equation if and only if $P$ is regular.
\end{lemma}

We note that for regular $n$-gons we have equality in the previous lemma. Also, in the Euclidean plane, the angle sum is always $\pi$.

Let $\beta_i$ be the angle $\angle\left(v_{i+\frac{n+1}{2}},v_i,p_i\right)=\angle\left(p_i,v_{i+\frac{n+1}{2}},v_i\right)$. Also, let $\gamma$ denote half of the inner angle of a regular triangle of minimal width $w$. We have the following:

\begin{lemma}\label{lemma:angles}
Let $P$ be an ordinary reduced $n$-gon of minimal width $w$. With the notations introduced above, we have $\beta_i\leq\gamma\leq\alpha_i$ with equality if and only if $P$ is a regular triangle.
\end{lemma}

\begin{proof}
First, we calculate the side length $a$ of the regular triangle in terms of $w$. By the hyperbolic Pythagorean theorem, we have
\begin{equation}\label{eq:pyth:regtriangle}
\cosh a=\cosh\frac{a}{2}\cosh w.
\end{equation}
The identity
$$
\cosh a=2\cosh^2\frac{a}{2}-1
$$
combined with \eqref{eq:pyth:regtriangle} leads to a quadratic equation for $\cosh\frac{a}{2}$, whose positive solution is
\begin{equation}\label{eq:sideofregtriangle}
\cosh\frac{a}{2}=\frac{\cosh w+\sqrt{\cosh^2 w+8}}{4}.
\end{equation}
Now we apply the hyperbolic law of sines for half of the regular triangle:
\begin{equation}\label{eq:singamma}
\sin\gamma=\frac{\sinh\frac{a}{2}}{\sinh a}=\frac{1}{2\cosh\frac{a}{2}}=\frac{-\cosh w+\sqrt{\cosh^2 w+8}}{4}.
\end{equation}
%As a consequence of Lemma \ref{lemma:congruence}, in the right triangle $\left[v_i,t_i,v_{i+\frac{n+1}{2}}\right]$, the hyperbolic law of sines gives
%\begin{equation}
%\frac{1}{\sinh d\left(v_i,v_{i+\frac{n+1}{2}}\right)}=\frac{\sinh\left(\alpha_i+\beta_i\right)}{\sinh w}.
%\end{equation}
From the hyperbolic law of cosines applied for the right triangle $\left[v_i,t_i,v_{i+\frac{n+1}{2}}\right]$ and also for the half of the regular triangle, we have
\begin{equation}\label{eq:coshw}
\cosh w=\frac{\cos 2\gamma}{\sin\gamma}=\frac{\cos\left(\alpha_i+\beta_i\right)}{\sin\beta_i}.
\end{equation}
We know that $\beta_i\leq\alpha_i$ (this is Theorem 2 (iii) in Lassak's paper \cite{Las24}), so from \eqref{eq:coshw} we imply
\begin{equation}
\frac{\cos 2\alpha_i}{\sin\alpha_i}\leq\frac{\cos 2\gamma}{\sin\gamma}\leq\frac{\cos 2\beta_i}{\sin\beta_i}.
\end{equation}
Finally, we observe that
$$
\frac{\cos 2x}{\sin x}=\frac{1-2\sin^2 x}{\sin x},
$$
and that the function $\frac{1-2x^2}{x}$ strictly monotonically decreases. Hence, $\sin\beta_i\leq\sin\gamma\leq\sin\alpha_i$, which in turn concludes the proof as all of these angles are acute. The case of equality is clear.
\end{proof}

Lassak also proved \cite{Las24} the following:

\begin{lemma}\label{lemma:diametralchord}
Let $D$ be the diameter of an ordinary reduced $n$-gon in $H^2$. Then, the endpoints of a chord of length $D$ are vertices of the polygon, where if $v_i$ is one of the endpoints of the chord, the other endpoint is either $v_{i+\frac{n-1}{2}}$ or $v_{i+\frac{n+1}{2}}$.
\end{lemma}

\section{The perimeter of ordinary reduced polygons}\label{sec:perim}

With the same method that was used in Sagmeister \cite{S24}, we can also investigate the extremality of the perimeter. With the notations introduced in the previous sections, let
$$
g_w\left(x\right)=\frac{1+\cos x-\sqrt{\left(1+\cos x\right)^2-4\tanh^2 w\cos x}}{2\tanh w}
$$
and
$$
%p_w\left(x\right)=\arsinh\left(\sin x\cdot\frac{\tanh w-g_w\left(x\right)}{\sqrt{\left(1-\lambda^2\right)\left(1-g_w\left(x\right)^2\right)}}\right).
p_w\left(x\right)=\arcosh\left(\frac{1- g_w\left(x\right)\tanh w}{\sqrt{1-\tanh^2 w}}\right).
$$
Then, we have the following formula for the perimeter.

\begin{theorem}\label{thm:perimeter:hyp}
Let $P$ be an ordinary reduced $n$-gon of thickness $w$ described as above. Then,
$$
\perim\left(P\right)=2\sum_{i=1}^n p_w\left(\varphi_i\right).
$$
\end{theorem}

\begin{proof}
By the definition of ordinary reduced polygons, $t_i$ is in the relative interior of the side $\left[v_{i+\frac{n-1}{2}},v_{i+\frac{n-1}{2}}\right]$, so
$$
\perim\left(P\right)=\sum_{i=1}^n d\left(v_i,v_{i+1}\right)=\sum_{i=1}^n \left(d\left(v_i,t_{i+\frac{n-1}{2}}\right)+d\left(t_i,v_{i+\frac{n+1}{2}}\right)\right).
$$
On the other hand, $d\left(v_i,t_{i+\frac{n+1}{2}}\right)=d\left(t_i,v_{i+\frac{n+1}{2}}\right)$ by Lemma~\ref{lemma:congruence}, so
$$
\perim\left(P\right)=2\sum_{i=1}^n d\left(t_i,v_{i+\frac{n+1}{2}}\right).
$$
%Using the law of sines and \eqref{eq:sinhci}, we get
%$$
%\sinh d\left(t_i,v_{i+\frac{n+1}{2}}\right)=\sin\varphi_i\cdot\sinh c_i=\sin\varphi_i\cdot\frac{\tanh w-g_w\left(\varphi_i\right)}{\sqrt{\left(1-\tanh^2 w\right)\left(1-g_w\left(\varphi_i\right)^2\right)}},
%$$
%concluding the proof.
If $b_i=d\left(p_i,t_i\right)$ and $c_i=d\left(p_i,v_{i+\frac{n+1}{2}}\right)$, $\tanh c_i$ can be expressed as
$$
\tanh c_i=\frac{\tanh w-g_w\left(\varphi_i\right)}{1-g_w\left(\varphi_i\right)\tanh w},
$$
where we refer to Sagmeister \cite{S24}. Hence,
$$
c_i=\artanh\left(\frac{\tanh w-g_w\left(\varphi_i\right)}{1-g_w\left(\varphi_i\right)\tanh w}\right).
$$
Using the identity
$$
\cosh\left(\artanh x\right)=\frac{1}{\sqrt{1-x^2}}
$$
and the hyperbolic Pythagorean theorem, we get
$$
\cosh d\left(t_i,v_{i+\frac{n+1}{2}}\right)=\frac{\cosh c_i}{\cosh b_i}=\frac{\left(1-g_w\left(\varphi_i\right)\tanh w\right)\sqrt{1-g_w^2\left(\varphi_i\right)}}{\sqrt{\left(1-g_w^2\left(\varphi_i\right)\right)^2-\left(\tanh w-g_w\left(\varphi_i\right)\right)^2}}=\frac{1- g_w\left(x\right)\tanh w}{\sqrt{1-\tanh^2 w}},
$$
concluding the proof.
\end{proof}

Now we are ready to prove the following.

\begin{theorem}\label{thm:monotone_convex}
%Let $P\subset H^2$ be an ordinary reduced $n$-gon of thickness $w$, and let $\widetilde{P}$ be a regular $n$-gon of thickness $w$. Then,
%$$
%\perim\left(P\right)\geq\perim\left(\widetilde{P}\right)
%$$
%with equality if and only if $P$ and $\widetilde{P}$ are congruent.
The function $p_w$ is strictly monotonically increasing and strictly convex on the interval $\left(0,\pi\right)$.
\end{theorem}

\begin{proof}
It is convenient to use the notation $r_w\left(x\right)=\sqrt{\left(1+\cos x\right)^2-4\tanh^2 w \cos x}$, so
\begin{equation}\label{eq:rprime}
r_w'\left(x\right)=\frac{-\sin x}{r_w\left(x\right)}\cdot\left(1+\cos x-2\tanh^2 w\right)
\end{equation}
and
$$
g_w'\left(x\right)=\frac{-\sin x}{r_w\left(x\right)}\cdot\left(\tanh w-g_w\left(x\right)\right).
$$
We deduce
\begin{gather*}
p_w'\left(x\right)=\sqrt{\frac{\tanh w}{\tanh w\cdot g_w\left(x\right)^2+\tanh w-2g_w\left(x\right)}}\cdot\frac{\sin x\left(\tanh w-g_w\left(x\right)\right)}{r_w\left(x\right)}=\\
=\frac{\sqrt{\tanh w\left(\tanh w-g_w\left(x\right)\right)\left(1+\cos x\right)}}{r_w\left(x\right)}=\frac{\cos\frac{x}{2}\sqrt{2\tanh w\left(\tanh w-g_w\left(x\right)\right)}}{r_w\left(x\right)}=\\
=\frac{\cos\frac{x}{2}\sqrt{2\tanh^2 w+r_w\left(x\right)-\left(1+\cos x\right)}}{r_w\left(x\right)}.
\end{gather*}
%and hence
%\begin{gather*}
%p_w''\left(x\right)=\frac{\cos\frac{x}{2}\left(-2\tanh^2 w+1+\cos x\right)r_w'\left(x\right)}{r_w^2\left(x\right)\sqrt{2\tanh^2 w+r_w\left(x\right)-\left(1+\cos x\right)}}+\\
%\frac{\sin\frac{x}{2}\left(-2\tanh^2 w+1+\cos x\right)-\cos\frac{x}{2}r_w'\left(x\right)+\sin x\cos\frac{x}{2}}{2r_w\left(x\right)\sqrt{2\tanh^2 w+r_w\left(x\right)-\left(1+\cos x\right)}}-\frac{\sin\frac{x}{2}}{2\sqrt{2\tanh^2 w+r_w\left(x\right)-\left(1+\cos x\right)}}.
%\end{gather*}
To calculate the second derivative, we use \eqref{eq:rprime} to substitute $r_w'\left(x\right)$, and we get to a common denominator. We also write $\sin x=2\sin\frac{x}{2}$ and $2\cos^2\frac{x}{2}=1+\cos x$. Thus, we obtain
%\begin{gather}\label{eq:numerator}
%\sin\frac{x}{2}\Bigg(-2\left(1+\cos x\right)\left(-2\tanh^2 w+\left(1+\cos x\right)\right)+\\
%\left(-2\tanh^2 w+\left(1+\cos x\right)\right)\left(\left(1+\cos x\right)^2-4\tanh^2 w\cos x\right)+\\
%\left(1+\cos x\right)\left(-2\tanh^2 w+\left(1+\cos x\right)\right)r_w\left(x\right)+\left(1+\cos x\right)\left(\left(1+\cos x\right)^2-4\tanh^2 w\cos x\right)-\\
%\left(\left(1+\cos x\right)^2-4\tanh^2 w\cos x\right)r_w\left(x\right)\Bigg).
%\end{gather}
%We can reformulate \eqref{eq:numerator} by observing that
%\begin{gather*}
%\left(1+\cos x\right)\left(-2\tanh^2 w+\left(1+\cos x\right)\right)-\left(\left(1+\cos x\right)^2-4\tanh^2 w\cos x\right)r_w\left(x\right)=\\
%=2\tanh^2 w\left(1-\cos x\right)r_w\left(x\right)
%\end{gather*}
%and
%\begin{gather*}
%-2\left(1+\cos x\right)\left(-2\tanh^2 w+\left(1+\cos x\right)\right)+\\
%\left(-2\tanh^2 w+\left(1+\cos x\right)\right)\left(\left(1+\cos x\right)^2-4\tanh^2 w\cos x\right)+\\
%\left(1+\cos x\right)\left(\left(1+\cos x\right)^2-4\tanh^2 w\cos x\right)=
%\end{gather*}
%that is positive for $x\in\left(0,\max_i\varphi_i\right]$. Jensen's inequality concludes the proof.
$$
p_w''\left(x\right)=\frac{\sin\frac{x}{2}\sqrt{2\tanh^2 w+r_w\left(x\right)-\left(1+\cos x\right)}}{r_w^3\left(x\right)}\cdot\left(\cos^2\frac{x}{2}r_w\left(x\right)+2\cos^2\frac{x}{2}\cos x-2\tanh^2 w\right).
$$
It is not too difficult to verify that this is positive; otherwise we can reorganize $\cos^2\frac{x}{2}r_w\left(x\right)+2\cos^2\frac{x}{2}\cos x-2\tanh^2 w$ as an inequality that is quadratic in $\tanh^2 w$, and we get that if $p_w''$ is not positive, $\tanh^2 w$ is either negative, or greater than 1, but both are impossible. Hence, $p_w$ is strictly convex.%, and the rest of the argument follows from Jensen's inequality.
%Theorem \ref{thm:perimeter:hyp}, the strict convexity of $p_w$ and Jensen's inequality implies
%$$
%\perim\left(P\right)\geq 2n p_w\left(\frac{\sum_{i=1}^n \varphi _i}{n}\right).
%$$
%If the perimeter is minimal, by the strict convexity of $p_w$, all $\varphi_i$ should be the same. However, it follows from Lemma \ref{lemma:congruence} that in this case $P$ is both equilateral and equiangular, and hence regular.
\end{proof}

The usual argument used on the Euclidean and spherical planes to find the reduced polygon with the minimal perimeter (see Lassak \cite{Las90} and Liu--Chang \cite{LC22}) uses Jensen's inequality after deriving a similar formula for the perimeter as in Theorem \ref{thm:perimeter:hyp}. However, if we consider the results of Lemma \ref{lemma:anglesum} and Theorem \ref{thm:monotone_convex}, we find that this approach does not work on the hyperbolic plane. Considering the perimeter of random ordinary reduced polygons given by Theorem \ref{thm:perimeter:hyp}, we have the following conjecture.

\begin{conj}
Let $P\subset H^2$ be an ordinary reduced $n$-gon of minimal width $w$. Then,
$$\perim\left(P\right)\leq\perim\left(\widetilde{P}\right)$$
with equality if and only if $P$ is regular.
\end{conj}

We note that contrary to the Euclidean and spherical planes, the perimeter of the regular $n$-gons of minimal width $w$ is not necessarily monotone in $n$. Depending on $w$, even the regular triangle can have a larger perimeter than the circle of the same minimal width.

%Finally, we will prove a similar result for reduced polygons in the Euclidean plane. Recall that in $\R^2$, a polygon is reduced if and only if it is ordinary reduced. For the Euclidean distance of $x,y\in\R^2$, we will use the notation $\left\|x-y\right\|$.

%\begin{theorem}
%Let $P\subset\R^2$ a reduced $n$-gon of thickness $w$, and let $\widetilde{P}$ be a reduced $n$-gon of the same thickness. Then,
%$$
%\perim\left(P\right)\geq\perim\left(\widetilde{P}\right)
%$$
%with equality if and only if $P$ and $\widetilde{P}$ are congruent.
%\end{theorem}

%\begin{proof}
%Let us use the same notations as in the hyperbolic case for the points and angles appearing in $P$. Then the right triangles $\left[v_i,t_{i+\frac{n+1}{2}},v_{i+\frac{n+1}{2}}\right]$ and $\left[v_i,t_i,v_{i+\frac{n+1}{2}}\right]$ are clearly congruent since they have a common hypotenuse and a leg of length $w$, and then $\left\|t_i-v_{i+\frac{n+1}{2}}\right\|=\left\|v_i-t_{i+\frac{n+1}{2}}\right\|$, and hence
%$$
%\left\|t_i-v_{i+\frac{n+1}{2}}\right\|=w\cdot\tan\frac{\varphi_i}{2}
%$$
%by Theorem 3 of Lassak's paper \cite{Las90}, so the perimeter of $P$ is
%$$
%\perim\left(P\right)=2w\sum_{i=1}^n \tan\frac{\varphi_i}{2}.
%$$
%Since $\tan$ is strictly convex on $\left(0,\frac{\pi}{2}\right)$ and $\sum_{i=1}^n \varphi_i=\pi$, we have
%$$
%\perim\left(P\right)\geq 2wn\tan\frac{\pi}{2n}=\perim\left(\widetilde{P}\right)
%$$
%by Jensen's inequality. By the strict convexity of the tangent function, the case of equality is also clear.
%\end{proof}

In a recent work of Chen, Hou and Jin \cite{CHJ23}, a consequence of the minimality of the perimeter of the regular $n$-gon of reduced $n$-gons of minimal width less than $\frac{\pi}{2}$ on the sphere is that the diameter and the circumradius of reduced spherical $n$-gons is minimized by the regular ones. Their argument gives the same result on the Euclidean plane. However, as we have seen, on the hyperbolic plane, we do not have the same result on the perimeter of ordinary reduced polygons. We propose the following question regarding the diameter and circumradius of ordinary reduced $n$-gons:\\

\noindent\textbf{Question. }Is the diameter and the circumradius of the regular $n$-gon extremal among ordinary reduced $n$-gons of the same minimal width on the hyperbolic plane?

\section{The diameter of ordinary reduced polygons}\label{sec:diameter}

In this section, we strengthen Lassak's Theorem \ref{thm:diameterbound:lassak} by using Lemma \ref{lemma:angles} in order to find the infimum of the circumradii of ordinary reduced polygons of a prescribed width. We have the following result.

\begin{theorem}
Let $P$ be an ordinary reduced polygon in $H^2$. Then,
$$
\diam\left(P\right)\leq 2\arcosh\left(\frac{\cosh w+\sqrt{\cosh^2 w+8}}{4}\right)
$$
with equality if and only if $P$ is a regular triangle.
\end{theorem}
\begin{proof}
Lemma \ref{lemma:diametralchord} allows us to assume that $\diam P=d\left(v_i,v_{i+\frac{n-1}{2}}\right)$ for some $i$, and we consider the right triangle $\left[v_i,t_i,v_{i+\frac{n+1}{2}}\right]$. Recall the notations $\alpha_i=\angle\left(t_i,v_{i+\frac{n+1}{2}},t_{i+\frac{n+1}{2}}\right)$ and $\beta_i=\angle\left(v_{i+\frac{n+1}{2}},v_i,t_i\right)$, and that $\angle\left(t_i,v_{i+\frac{n+1}{2}},v_i\right)=\alpha_i+\beta_i$ as a consequence of Lemma \ref{lemma:congruence}. Since $d\left(v_i,t_i\right)=w$, we get
\begin{equation}\label{eq:sinhoftheleg}
\sinh d\left(t_i,v_{i+\frac{n+1}{2}}\right)=\frac{\sinh w\cdot\sin\alpha_i}{\sin\left(\alpha_i+\beta_i\right)}\leq\frac{\sinh w}{2\cos\beta_i}
\end{equation}
from the hyperbolic law of sines, and using the inequality $\alpha_i\geq\beta_i$ (cf. Lassak \cite{Las24}), the monotonicity of the $\sin$ function on the interval $\left(0,\frac{\pi}{2}\right)$ and the identity $\sin 2x=2\sin x\cos x$. From Lemma \ref{lemma:angles} we also have $\beta_i\leq\gamma$ where $\gamma$ denotes half the angle of a regular triangle of minimal width $w$, so we obtain
\begin{equation}\label{eq:sinhoftheleg:upperbound}
\sinh d\left(t_i,v_{i+\frac{n+1}{2}}\right)\leq\frac{\sinh w}{2\cos\gamma}.
\end{equation}
from \eqref{eq:sinhoftheleg}. Now we use \eqref{eq:sinhoftheleg:upperbound} and the identity $\cosh^2 x-\sinh^2 x=1$ to obtain
\begin{equation}
\cosh d\left(t_i,v_{i+\frac{n+1}{2}}\right)\leq\sqrt{1+\frac{\sinh^2 w}{4\cos^2\gamma}}.
\end{equation}
We now use the hyperbolic Pythagorean theorem with the assumption $d\left(v_i,v_{i+\frac{n+1}{2}}\right)=\diam P$:
\begin{equation}\label{eq:diambound}
\cosh\diam P=\cosh w \cosh d\left(t_i,v_{i+\frac{n+1}{2}}\right)\leq\cosh w\sqrt{1+\frac{\sinh^2 w}{4\cos^2\gamma}}.
\end{equation}
Clearly, from the equality case of Lemma \ref{lemma:angles} we imply that equality holds if and only if $P$ is a regular triangle. On the other hand, if $P$ is a regular triangle centered at $p$, in the right triangle $\left[p,t_1,v_3\right]$, the angles are $\frac{\pi}{3}$, $\frac{\pi}{2}$ and $\gamma$, respectively, while the length of leg opposite to the $\frac{\pi}{3}$ angle is $\frac{\diam P}{2}$, so using the well-know identity $\cosh b=\frac{\cos B}{\sin A}$ for hyperbolic right triangles of acute angles $A$ and $B$ and legs $a$ and $b$ respectively, we obtain
$$
\cosh\frac{\diam P}{2}=\frac{1}{2\sin\gamma}.
$$
This together with \eqref{eq:diambound} gives
$$
\arcosh\left(\cosh w\sqrt{1+\frac{\sinh^2 w}{4\cos^2\gamma}}\right)=2\arcosh\frac{1}{2\sin\gamma}.
$$
Finally, \eqref{eq:singamma} concludes the proof.
\end{proof}

From this sharp upper bound we can see that Lassak's following conjecture is true.

\begin{corollary}
Let $P$ be an ordinary reduced polygon in $H^2$. Then,
$$
1<\frac{\diam P}{w}<2.
$$
\end{corollary}

\begin{proof}
The first inequality trivially holds as $\diam P$ is the maximal width of $P$, and equality holds exactly for bodies of constant width, which are h-convex, so no polygon is of constant width (see Böröczky, Csépai, Sagmeister \cite{BoCsS} for further details).

The second inequality is an easy consequence of Theorem \ref{thm:diameterbound:sagmeister} and the strict monotonicity of $\cosh$ for $w>0$.

We can also observe, that both of these constant bounds are the optimal ones: on one hand, regular $\left(2k+1\right)$-gons are ordinary reduced polygons that approximate a disk as $k\to\infty$, while on the other hand we also have
\begin{gather*}
\lim_{w\to +\infty}\frac{2\arcosh\left(\frac{\cosh w+\sqrt{\cosh^2 w+8}}{4}\right)}{w}=2\cdot\lim_{w\to +\infty}\frac{\arcosh\left(\frac{\cosh w}{2}\right)}{w}=2\cdot\lim_{w\to +\infty}\frac{\arcosh\left(\frac{e^w}{4}\right)}{w}=\\
=2\cdot\lim_{w\to +\infty}\frac{\ln\left(\frac{e^w}{4}\right)}{w}=2.
\end{gather*}
\end{proof}

\section{The circumradius of ordinary reduced polygons}\label{sec:circumradius}

Now we answer another question of Lassak proposed in \cite{Las24} with the following upper bound for the circumradius.

\begin{theorem}
Let $P$ be an ordinary reduced polygon in $H^2$ of width $w$. Then,
$$
R\left(P\right)\leq\arsinh\left(\frac{2}{\sqrt{3}}\sqrt{\left(\frac{\cosh w+\sqrt{\cosh^2 w+8}}{4}\right)^2-1}\right)
$$
with equality if and only if $P$ is a regular triangle.
\end{theorem}
\begin{proof}
The hyperbolic Jung theorem (see Dekster \cite{Dek95J}) says that
$$
R\left(P\right)\leq\arsinh\left(\frac{2}{\sqrt{3}}\sinh\left(\frac{\diam\left(P\right)}{2}\right)\right).
$$
Theorem \ref{thm:diameterbound:sagmeister}, and the identity
$$
\sinh\left(\arcosh\left(x\right)\right)=\sqrt{x^2-1}
$$
concludes the proof. The case of the equality is clear form the equality case of Theorem \ref{thm:diameterbound:sagmeister}.
\end{proof}

As a consequence, we have the following.

\begin{corollary}\label{corollary:Rislessthanw}
Let $P$ be an ordinary reduced polygon in the hyperbolic plane. Then,
$$
\frac{1}{2}<\frac{R\left(P\right)}{w}<1.
$$
\end{corollary}

\begin{proof}
The first inequality follows from the monotonicity of the minimal width: the minimal width of the circumcircle is $2R\left(P\right)$, but the circumcircle is reduced, so $w<2R\left(P\right)$.

The second inequality is equivalent with
$$
\sqrt{\left(\frac{\cosh w+\sqrt{\cosh^2 w+8}}{4}\right)^2-1}<\frac{\sqrt{3}}{2}\sinh w.
$$
After taking the square of both sides, and applying the identity $\sinh^2 x=\cosh^2 x-1$, the inequality we want to verify takes the form
$$
\left(\frac{\cosh w+\sqrt{\cosh^2 w+8}}{4}\right)^2<\frac{3}{4}\cosh^2 w+\frac{1}{4}.
$$
With a few simple steps, we can reorganize this inequality as
$$
0<6\cosh^4 w-7\cosh^2 w+1=\left(6\cosh^2 w-1\right)\left(\cosh^2 w-1\right),
$$
which clearly holds.

Let us observe that these constants provide the best possible bounds for the ratio of the circumradius and the minimal width. The sharpness of the first inequality comes from considering a convergent sequence $P_k$ of regular $\left(2k+1\right)$-gons of minimal width $w$, whose limit is a disk of width $w$ and radius $2w$. As for the second inequality, we can see that
\begin{gather*}
\lim_{w\to+\infty}\frac{\arsinh\left(\frac{2}{\sqrt{3}}\sqrt{\left(\frac{\cosh w+\sqrt{\cosh^2 w+8}}{4}\right)^2-1}\right)}{w}=\lim_{w\to+\infty}\frac{\arsinh\left(\frac{2}{\sqrt{3}}\sqrt{\left(\frac{e^w}{4}\right)^2-1}\right)}{w}=\\
=\lim_{w\to+\infty}\frac{\arsinh\left(\frac{e^w}{2\sqrt{3}}\right)}{w}=\lim_{w\to+\infty}\frac{\ln\left(\frac{e^w}{\sqrt{3}}\right)}{w}=1.
\end{gather*}
\end{proof}

In the Euclidean plane, Fabi\'nska \cite{Fab} proved that for an arbitrary reduced polygon of width $w$ there is some boundary point, such that the circular disk of radius $w$ centered at that boundary point covers the disk. Musielak \cite{Mus} showed that the same holds for spherical reduced polygons. The remainder of this section is dedicated to the proof of the same statement for ordinary reduced polygons in the hyperbolic plane. In most of the steps we can repeat Musielak's spherical argument, but we also apply a few minor adjustments. First, we will need a few lemmas.

\begin{lemma}\label{lemma:rballscontainment}
For a compact set $X\subseteq H^n$ and a point $z\in\conv X$, we have
$$
\bigcap_{x\in X}B\left(x,r\right)\subseteq B\left(z,r\right)
$$
for any positive radius $r$.
\end{lemma}
\begin{proof}
If $z\in X$, the statement trivially holds, so we assume $z\in\conv X\setminus X$. If we consider the Beltrami--Cayley--Klein model, Euclidean and hyperbolic lie segments coincide, so it is easy to see that Minkowski's theorem holds, that is, $\conv X=\conv E\left(X\right)$ where $E\left(X\right)$ denotes the extreme points of $\conv X$ We also have $E\left(X\right)\subseteq X$, so clearly it is sufficient to prove
\begin{equation}\label{eq:ballintersection}
\bigcap_{x\in E}B\left(x,r\right)\subseteq B\left(z,r\right)
\end{equation}
for some set $E\subseteq E\left(X\right)$.

Considering again the linearity preserving properties of the Beltrami--Cayley--Klein model, by Carathéodory's theorem there are $k\leq n+1$ points $e_1,\ldots,e_k$ in $E\left(X\right)$ such that $z\in\left[e_1,\ldots,e_k\right]$. Naturally, we can assume that $\bigcap_{i=1}^k B\left(e_i,r\right)$ is not empty, otherwise the statement is trivial.

It is easy to see that for a compact set $Y$, the function $d\left(y,\,\cdot\,\right)$ restricted to $Y$ attains its maximum for some point $e\in E\left(Y\right)$. Therefore, if we choose some point $y\in \bigcap_{i=1}^k B\left(e_i,r\right)$, and we set $Y=\left\{e_1,\ldots,e_k,z\right\}$, then
$$
d\left(y,z\right)\leq\max_i d\left(y,e_i\right)\leq r,
$$
so $y\in B\left(z,r\right)$, and that concludes the proof.
\end{proof}

For a convex body $K\subset H^n$ and a positive number $r$, let us introduce the notation $C\left(K,r\right)$ for the set of centers such that the closed balls of radius $r$ centered at points in this set contain $K$, i.e.
$$
C\left(K,r\right)=\left\{x\in H^n\colon K\subseteq B\left(x,r\right)\right\}.
$$
Musielak's following spherical characterization \cite{Mus} still holds in $H^n$. We will omit the proof, as his argument only uses metric considerations, and hence it remains true in $H^n$.

\begin{lemma}\label{lemma:characterizationofcenterset}
For a convex body $K\subset H^n$ and a positive number $r$,
$$
C\left(K,r\right)=\bigcap_{e\in E\left(K\right)}B\left(e,r\right).
$$
\end{lemma}

Similarly, the following lemma of Musielak \cite{Mus} remains true in the hyperbolic plane as well. His inductive proof can be repeated to the letter, as hyperbolic balls of radius $r$ are also $r$-spindle convex (e.g. this is a trivial corollary of Theorem 1.2 in Böröczky, Csépai, Sagmeister \cite{BoCsS}).

\begin{lemma}\label{lemma:boundaryofintersectionofrballs}
If $K\subset H^n$ is a convex body obtained as the intersection of finitely many circular disks of radius $r$, then the boundary of $K$ is the union of finitely many shorter circular arcs of radius $r$ such that they all have different centers.
\end{lemma}

For a convex body $K$, we also introduce the notation $E^*\left(K\right)$ as
$$
E^*\left(K\right)=\left\{e\in E\left(K\right)\colon\left|\partial C\left(K,r\right)\cap\partial B\left(e,r\right)\right|>1\right\},
$$
where $\partial X$ stands for the boundary of $X$ and $\left|X\right|$ for its cardinality. Then, from Lemma \ref{lemma:characterizationofcenterset} and Lemma \ref{lemma:boundaryofintersectionofrballs} we can immediately derive the following.

\begin{corollary}\label{corollary:Estarcharacterization}
Let $P\subset H^2$ be a convex polygon. Then,
$$
C\left(P,r\right)=\bigcap_{e\in E^*\left(P\right)}B\left(e,r\right).
$$
\end{corollary}

Finally, we have all the tools to prove the following.

\begin{theorem}
Let $P\in H^2$ be an ordinary reduced polygon of minimal width $w$. Then, there is a boundary point $z\in\partial P$, such that $P\subset B\left(z,w\right)$.
\end{theorem}

\begin{proof}
Assuming that $P$ is an $n$-gon, we use the same notations as in Section \ref{sec:ordinary}. We consider the set $C\left(P,w\right)$, and we aim to show that the intersection $C\left(P,w\right)\cap\partial P$ is not empty. We prove this by contradiction.

Corollary \ref{corollary:Rislessthanw} and the hyperbolic Jung theorem (cf. Dekster \cite{Dek95J}) implies that $C\left(P,w\right)$ intersects the interior of $P$. If $C\left(P,w\right)$ also contains some point in the exterior of $P$, then the proof is complete by the convexity of $C\left(P,w\right)$ (see Lemma \ref{lemma:characterizationofcenterset}). So let us assume that $C\left(P,w\right)$ is a subset of the interior of $P$.

Let $e_1,\ldots,e_m$ be the points of $E^*\left(P\right)$ in a positive orientation, where we understand the indices modulo $m$. By the assumption that $C\left(P,w\right)$ lies in the interior of $P$, it is easy to see from the definition of ordinary reduced polygons that $3\leq m\leq n$. By Lemma \ref{lemma:boundaryofintersectionofrballs}, to each vertex $e_i$ in $E^*\left(P\right)$, there is a shorter circular arc $\mathcal{C}_i$ of $B\left(e_i,r\right)$ on the boundary. Let $q_i$ be the intersection of $\mathcal{C}_i$ and $\mathcal{C}_{i+1}$. Lemma \ref{lemma:boundaryofintersectionofrballs} and Corollary \ref{corollary:Estarcharacterization} implies that the boundary of $C\left(P,w\right)$ is the union of short circular arcs of radius $r$ connecting $q_i$ and $q_{i+1}$ for $1\leq i\leq m$; let us denote these arcs by $\arc{q_iq_{i+1}}$.

Each point $e_i$ of $E^*\left(P\right)$ is a vertex $v_{\sigma\left(i\right)}$ for an injective map $\sigma\colon\left\{1,\ldots,m\right\}\to\left\{1,\ldots,n\right\}$. We set $s_i=t_{\sigma\left(i\right)}$ for $i\in\left\{1,\ldots,m\right\}$. By our assumption, $s_i\in\partial B\left(e_i,r\right)$ is not in $C\left(P,w\right)$, and hence it is not contained in the arc $\arc{q_i,q_{i+1}}$.

Clearly, the points $e_1$, $q_2$ and $s_1$ are not collinear, since $B\left(e_1,w\right)\supset P$, both $q_2$ and $s_1$ are boundary points of the circle $B\left(e_1,w\right)$, on the other hand, while $s_1$ is a boundary point of $P$, $q_2$ lies in the interior. Hence, $s_1$ is either in the same open hemisphere bound by the line through $e_1$ and $q_2$ as $q_3$, or as $q_m$ (or equivalently as $e_2$). By symmetry, we can assume without loss of generality that the first case occurs.

We also consider a diametral chord $\left[x,y\right]$ of $P$. We note that $x=v_i$ and $y=v_{i+\frac{n\pm 1}{2}}$ for some $1\leq i\leq n$ where in this case the indices are understood modulo $n$ (cf. Lassak \cite{Las24}). Let $x'$ and $y'$ be the orthogonal projections of $x$ and $y$ to the opposite sides, respectively (i.e. $x'=t_i$ and $y'=t_{i+\frac{n\pm 1}{2}}$). Let $m_1$ be the midpoint of $\left[x,y'\right]$ and $m_2$ be the midpoint of $\left[x',y\right]$. Then, if we consider the boundary of $P$, one of the triples $\left(x,m_1,y'\right)$ and $\left(x',m_2,y\right)$ are on the boundary in this order with the same orientation, we assume without loss of generality that this is the positive orientation.

We can also assume that the triple $\left(e_1,x,e_2\right)$ are on the boundary of $P$ in this order with respect to the positive orientation, where we allow the case $x=e_1$. This implies that $\left(s_1,x',s_2\right)$ also have the same order in the positive orientation, as all chords $\left[v_i,t_i\right]$ half the perimeter (see Lassak \cite{Las24} or Sagmeister \cite{S24}).

By our assumption that $C\left(P,w\right)$ is in the interior of $P$, so $x'$ is not contained in $B\left(e_k,r\right)$ for some $2\leq k\leq m$ (by the assumption on the position of $p_3$ and $s_1$, we deduce $d\left(e_1,x'\right)\leq d\left(x,x'\right)=w$). We can also observe that there is such $e_k$ on the opposite half-plane of the line through $e_1$ and $p_2$ as $x'$. Let $z$ be the interior point of the segment $\left[x',e_k\right]$ such that $d\left(x',z\right)=d\left(x',x\right)=w$. In particular, $z\in P$. Considering the two triangles $\left[z,x',y\right]$ and $\left[x,x',y\right]$, they have two equal sides with different enclosed angle. This implies
$$
d\left(z,y\right)>d\left(x,y\right)=\diam P,
$$
which is clearly a contradiction.
\end{proof}

\end{document}